\documentclass[11pt]{amsart}

\usepackage{amsmath,amssymb,amsthm,amsfonts,amscd,flafter,epsf, epsfig,graphicx,verbatim,pinlabel,mathrsfs}
\usepackage[all]{xy}
\usepackage{epsf}

\title{Contact monoids and Stein cobordisms}
\author[John A. Baldwin]{John A. Baldwin}
\address {Department of Mathematics, Princeton University\\ Princeton, NJ 08544-1000}
\email {baldwinj@math.princeton.edu}
\thanks{The author was partially supported by an NSF Postdoctoral Fellowship.}
\date{}

\newcommand\bH{\mathbf{H}}

\newcommand\hf{\widehat{HF}}

\newcommand\zzt{\mathbb{Z}_2}

\newcommand\sminus{-}
\newcommand\LL{\mathbb{L}}
\newcommand\LLb{\mathbb{L}_{\beta}}

\newtheorem{theorem}{Theorem}[section]

\newtheorem{corollary}[theorem]{Corollary}
\newtheorem{proposition}[theorem]{Proposition}
\newtheorem{question}[theorem]{Question}

\theoremstyle{definition}

\newtheorem{remark}[theorem]{Remark}

\begin{document}
\begin{abstract} 
Suppose $S$ is a compact surface with boundary, and let $\phi$ be a diffeomorphism of $S$ which fixes the boundary pointwise. We denote by $(M_{S,\phi},\xi_{S,\phi})$ the contact 3-manifold compatible with the open book $(S,\phi)$. In this article, we construct a Stein cobordism from the contact connected sum $(M_{S,h},\xi_{S,h})\,\#\,(M_{S,g},\xi_{S,g})$ to $(M_{S,hg},\xi_{S,hg})$, for any two boundary-fixing diffeomorphisms $h$ and $g$. This cobordism accounts for the comultiplication map on Heegaard Floer homology discovered in \cite{bald3}, and it illuminates several geometrically interesting monoids in the mapping class group $Mod^+(S,\partial S).$ We derive some consequences for the fillability of contact manifolds obtained as cyclic branched covers of transverse knots.\end{abstract}

\maketitle

\section{Introduction}
Let $M$ be a closed, oriented 3-manifold. In \cite{giroux2}, Giroux proves that there is a one-to-one correspondence between isotopy classes of contact structures on $M$ and open book decompositions of $M$ up to an equivalence called positive stabilization. Giroux's work places contact geometry on a more topological footing, and allows us to translate questions about tightness and fillability of contact structures into questions about diffeomorphisms of compact surfaces with boundary. In some cases, this translation is rather well-understood. For instance, Akbulut and Ozbagci \cite{AO} and, independently, Giroux \cite{giroux}, have shown that a contact manifold $(M, \xi)$ is Stein fillable if and only if $(M,\xi)$ is supported by \emph{some} open book $(S,\phi)$ for which $\phi$ is a product of right-handed Dehn twists. In a similar vein, Honda, Kazez and Mati{\'c}, generalizing a result of Goodman \cite{good}, prove that $(M,\xi)$ is tight if and only if \emph{every} open book $(S,\phi)$ supporting $(M,\xi)$ has \emph{right-veering} monodromy $\phi$ \cite{hkm1}.

Let $Mod^+(S,\partial S)$ denote the set of isotopy classes of orientation-preserving diffeomorphisms of $S$ which restrict to the identity on $\partial S$. The subset $Dehn^+(S,\partial S)\subset Mod^+(S,\partial S)$ whose elements are represented by compositions of right-handed Dehn twists is closed under composition and contains the isotopy class of the identity; in other words, it is a monoid. So, too, is the subset $Veer(S,\partial S)$ consisting of isotopy classes of right-veering diffeomorphisms of $S$. One is therefore tempted to ask whether the subsets of $Mod^+(S,\partial S)$ corresponding to Stein fillable, or, respectively, tight contact manifolds form monoids as well. In this article, we show that the first of these sets is a monoid and we point out two more geometrically significant ``contact monoids" in $Mod^+(S,\partial S)$, thereby strengthening our understanding of the link between contact structures and open books. 

Let $(M_{S,\phi}, \xi_{S,\phi})$ denote the contact manifold supported by the open book $(S,\phi)$. Most of the results in this paper stem from the following theorem.

\begin{theorem}
\label{thm:Stein}
$(M_{S,hg},\xi_{S,hg})$ is the result of contact $(-1)$-surgery on a Legendrian link $\mathbb{L}$ in the contact connected sum $(M_{S,h},\xi_{S,h})\,\#\,(M_{S,g},\xi_{S,g})$.
\end{theorem}

According to Eliashberg, the cobordism from $(M_{S,h},\xi_{S,h})\,\#\,(M_{S,g},\xi_{S,g})$ to $(M_{S,hg},\xi_{S,hg})$ obtained by attaching contact $(-1)$-framed 2-handles to the components of the link $\mathbb{L}$ in Theorem \ref{thm:Stein} carries a natural Stein structure which is compatible with the contact structures on either end \cite{yasha3}. So, we may alternatively think of Theorem \ref{thm:Stein} as the statement that there exists a Stein 2-handle cobordism from the contact connected sum $(M_{S,h},\xi_{S,h})\,\#\,(M_{S,g},\xi_{S,g})$ to $(M_{S,hg},\xi_{S,hg})$. The theorem below is a straightforward consequence of Theorem \ref{thm:Stein}.

\begin{theorem}
\label{thm:propp}
Suppose that $\bH$ is some property of contact manifolds which is preserved under the operations of contact connected sum and contact $(-1)$-surgery on Legendrian links. Then the set of $\phi\in Mod^+(S,\partial S)$ for which $(M_{S,\phi},\xi_{S,\phi})$ satisfies the property $\bH$ is closed under composition.
\end{theorem}

To see this, suppose that $(M_{S,h},\xi_{S,h})$ and $(M_{S,g},\xi_{S,g})$ satisfy property $\bH$. Then so does the contact connected sum $(M_{S,h},\xi_{S,h})\,\#\,(M_{S,g},\xi_{S,g})$. Theorem \ref{thm:Stein} therefore implies that $(M_{S,hg},\xi_{S,hg})$ satisfies property $\bH$ as well. 

Examples of such properties $\bH$ are Stein fillability, as well as strong and weak symplectic fillability \cite{yasha3,weinstein,EH3}. In particular, let $Stein(S,\partial S)$, $Strong(S,\partial S)$ and $Weak(S,\partial S)$ denote the subsets of $Mod^+(S,\partial S)$ whose elements give rise to open books compatible with Stein fillable, strongly symplectically fillable and weakly symplectically fillable contact manifolds, respectively. Then Theorem \ref{thm:propp} implies the following.

\begin{theorem}
\label{thm:monoid}
$Stein(S,\partial S)$, $Strong(S,\partial S)$ and $Weak(S,\partial S)$ are monoids.
\end{theorem}

The contact invariant in Heegaard Floer homology is well-behaved with respect to the maps induced by Stein cobordisms \cite{osz1}. Specifically, if $(M',\xi')$ is obtained from $(M,\xi)$ by performing contact $(-1)$-surgery on a Legendrian knot, and $W$ is the corresponding 2-handle cobordism from $M$ to $M'$, then the map $$F_{-W}:\hf(-M')\rightarrow \hf(-M)$$ sends $c(\xi')$ to $c(\xi)$. In addition, for two contact manifolds $(M_1,\xi_1)$ and $(M_2,\xi_2)$, the contact invariant $c(\xi_1 \# \xi_2)$ is identified with $c(\xi_1) \otimes c(\xi_2)$ via the isomorphism $$\hf(-(M_1 \# M_2)) \cong \hf(-M_1)\otimes_{\zzt} \hf(-M_2).$$ Coupled with these facts, Theorem \ref{thm:Stein} immediately reproduces the following result from \cite{bald3}.

\begin{corollary}[{\rm \cite[Theorem 1.4]{bald3}}]
\label{cor:comult}
There exists a ``comultiplication" map $$\hf(-M_{S,hg}) \rightarrow \hf (-M_{S,h})\otimes_{\zzt} \hf(-M_{S,g})$$ which sends $c(\xi_{S,hg})$ to $c(\xi_{S,h})\otimes c(\xi_{S,g})$.
\end{corollary} 

In particular, the set $OSz(S,\partial S)$ of $\phi \in Mod^+(S,\partial S)$ for which $c(\xi_{S,\phi}) \neq 0$ forms a monoid. This prompts the following question. (Below, $Tight(S,\partial S)$ is the subset of $Mod^+(S,\partial S)$ whose elements give rise to open books compatible with tight contact manifolds.)

\begin{question}
\label{ques:tight}
 Is $Tight(S,\partial S)$ a monoid?
\end{question}

Corollary \ref{cor:comult} does not provide an answer to Question \ref{ques:tight}, as there are tight contact structures whose contact invariants vanish \cite{lisstip, ghihonvh}. In fact, the question of whether tightness is preserved by contact $(-1)$-surgery on a Legendrian knot is open for closed contact manifolds.\footnote{Honda has found a tight contact structure on a genus 4 handlebody which becomes overtwisted after performing contact $(-1)$-surgery on a Legendrian knot \cite{honda3}, so the closedness condition is necessary when asking this question.} Interestingly, Theorem \ref{thm:propp} implies that this seemingly more basic question is actually equivalent to Question \ref{ques:tight}.

It is worth noting that, in general, these various subsets of $Mod^+(S,\partial S)$ are nested as follows (here, we have suppressed the notation for the surface $S$):
\begin{eqnarray}
\label{eqn:nest}
Dehn^+ \subsetneq Stein \subsetneq Strong \subsetneq Weak \subsetneq Tight \subsetneq Veer.
\end{eqnarray}  Moreover, 
\begin{eqnarray}
Strong\subsetneq OSz\subsetneq Tight,
\end{eqnarray} while $Weak$ and $OSz$ are incomparable in general \cite{ghiggini}. The strictness of the leftmost inclusion in (\ref{eqn:nest}) is a recent result due to Baker, Etnyre and Van Horn-Morris \cite{bev2} and, independently, Wand \cite{wand}, who construct open books for Stein fillable contact structures whose monodromies cannot be expressed as products of right-handed Dehn twists. The strictness of the rightmost inclusion is due to Honda, Kazez, and Mati{\'c}, who show any contact structure is supported by some open book with right-veering monodromy \cite{hkm1}. The other inclusions are more familiar and can be found in \cite{ghiggini3, yasha6, ghiggini4}.


The three contact monoids in Theorem \ref{thm:monoid} have been discovered independently by Baker, Etnyre and Van Horn-Morris \cite{bev2}, who construct their own Stein cobordism from $(M_{S,h},\xi_{S,h})\,\sqcup\,(M_{S,g},\xi_{S,g})$ to $(M_{S,hg},\xi_{S,hg})$. In fact, it was not until hearing of their result that I realized that the cobordism from $(M_{S,h},\xi_{S,h})\,\#\,(M_{S,g},\xi_{S,g})$ to $(M_{S,hg},\xi_{S,hg})$ defined implicitly in the last section of my paper with Plamenevskaya \cite{baldpla} carries a very natural Stein structure. 

The proof of Theorem \ref{thm:Stein} in this article makes use of standard tools in convex surface theory together with some small input from Heegaard Floer homology. In contrast, the approach of Baker, Etnyre and Van Horn-Morris involves an understanding of the contact structures associated to various cables of the binding of an open book. It would be interesting to determine whether our different approaches yield what are more or less the same Stein cobordisms in the end. 

The methods used in this paper can be applied in other settings as well. For example, suppose that $K$ is a transverse knot in the standard tight contact manifold $(S^3,\xi_{std})$. A well-known result of Bennequin asserts that $$sl(K) \leq -\chi(\Sigma),$$ where $sl(K)$ denotes the self-linking number of $K$ and $\Sigma$ is any Seifert surface for $K$ \cite{benn}. We say that $K$ \emph{realizes its Bennequin bound} if $sl(K) = -\chi(\Sigma)$ for some Seifert surface $\Sigma$. In \cite{hedd2}, Hedden proves that if $K$ is fibered and realizes its Bennequin bound then the open book associated to $K$ supports the contact manifold $(S^3,\xi_{std})$ (see \cite{bev} for a more general result). If $(S,\phi)$ denotes this open book, then $(S,\phi^n)$ supports the contact manifold obtained by taking the n-fold cyclic cover of $(S^3,\xi_{std})$ branched along $K$. Since $(S^3,\xi_{std})$ is Stein fillable, Theorem \ref{thm:monoid} implies the following.

\begin{corollary}
\label{cor:fbnc}
If $K$ is a fibered transverse knot in $(S^3,\xi_{std})$ which realizes its Bennequin bound, then the n-fold cyclic cover of $(S^3,\xi_{std})$ branched along $K$ is Stein fillable. \end{corollary}

\begin{remark} The statement in Corollary \ref{cor:fbnc} follows independently from the fact that if $K$ is a fibered transverse knot in $(S^3,\xi_{std})$ which realizes its Bennequin bound, then $K$ is strongly quasi-positive \cite{hedd3} and therefore bounds a complex curve $\Sigma$ in $B^4\subset \mathbb{C}^2$ \cite{rud}. Hence, the n-fold cyclic cover of $(B^4,i)$ branched along $\Sigma$ is a holomorphic filling of the n-fold cyclic cover of $(S^3,\xi_{std})$ branched along $K$. Lastly, a result of Bogomolov and de Oliveira tells us that this (indeed, any) holomorphic filling may be deformed into the blow-up of a Stein filling \cite{bdo}.\end{remark}

Using a slight variation of the main technique in this paper as suggested by Van Horn-Morris, combined with the ideas in \cite[Section 3]{bev}, we can prove a much stronger result which does not assume that $K$ is fibered.

\begin{theorem}
\label{thm:bnc}
If $K$ is a transverse knot in a Stein (resp. strongly/weakly symplectically) fillable contact manifold $(M,\xi)$ which realizes its Bennequin bound, then the n-fold cyclic cover of $(M,\xi)$ branched along $K$ is Stein (resp. strongly/weakly symplectically) fillable.
\end{theorem}

It bears mentioning that Theorem \ref{thm:bnc} overlaps with a similar result in \cite[Corollary 1.3]{baldpla}. There, we show, using entirely different methods, that if $K$ is a transverse knot in $(S^3,\xi_{std})$ which belongs to a \emph{quasi-alternating} knot type (see \cite{osz12}) and satisfies the equality \begin{equation}\label{eqn:benn}sl(K)=\sigma(K)-1,\end{equation} then the double cover of $(S^3,\xi_{std})$ branched along $K$ is tight. In general, transverse knots in $(S^3,\xi_{std})$ satisfy the Bennequin-like bound \cite{pla3,sh}, \begin{equation}\label{eqn:sls}sl(K) \leq s(K)-1,\end{equation} where $s(K)$ is the concordance invariant defined by Rasmussen using Khovanov homology in \cite{ras3}. In his paper, Rasmussen shows that $$|s(K)|\leq2g_4(K),$$ and uses this fact to give a combinatorial proof of the Milnor conjecture (here, $g_4(K)$ denotes the 4-ball genus of $K$). In particular, the inequality above implies that \begin{equation}\label{eqn:sig}s(K)-1 \leq -\chi(\Sigma)\end{equation} for any Seifert surface $\Sigma$ which bounds $K$. 

Now, if $K$ is quasi-alternating, then $s(K) = \sigma(K)$ \cite{ras3,manozs}. If, in addition, $K$ realizes its Bennequin bound, then Equations (\ref{eqn:sls}) and (\ref{eqn:sig}) force the equality in Equation (\ref{eqn:benn}), and both the result in \cite[Corollary 1.3]{baldpla} described above and Theorem \ref{thm:bnc} imply that the double cover of $(S^3,\xi_{std})$ branched along $K$ is tight.

\subsection*{Acknowledgements} I thank John Etnyre and Jeremy Van Horn-Morris for very helpful correspondence. I am especially grateful for Jeremy's comments on an early draft of this paper, which lead to Theorem \ref{thm:bnc}.

\section{Proof of Theorem \ref{thm:Stein}}
First, we describe the contact 3-manifold, $(M_{S,\phi},\xi_{S,\phi})$, which is compatible with the open book $(S,\phi)$. Let $U$ be the handlebody defined by $U=S \times [-1,1]/\sim$, where $(x,t)\sim(x,0)$ for all $x \in \partial S$ (see Figure \ref{fig:handle}). The oriented curve $\Gamma = \partial S \times \{0\}$ divides $\Sigma = \partial U$ into two pieces, $\Sigma^+ = S\times \{1\}$ and $\Sigma^- =  -S \times \{-1\}$. We may therefore view $\phi$ as a boundary-fixing diffeomorphism of $\Sigma^+$. Note that $\partial \Sigma^+ = \Gamma = -\partial \Sigma^-$, and let $r:\Sigma \rightarrow \Sigma$ be the orientation-reversing involution defined by reflection across $\Gamma$. 

\begin{figure}[!htbp]
\labellist 
\hair 2pt 
\tiny
\pinlabel $\Gamma$ at 479 130
\pinlabel $\partial S$ at 104 61
\pinlabel $\Sigma^+$ at 461 184
\pinlabel $\Sigma^-$ at 461 76
\pinlabel $-1$ at 151 27
\pinlabel $1$ at 157 233
\small
\pinlabel $U$ at 405 4
\pinlabel $S\times [-1,1]$ at 233 4
\pinlabel $S$ at 52 4
\endlabellist 
\begin{center}
\includegraphics[height = 5cm]{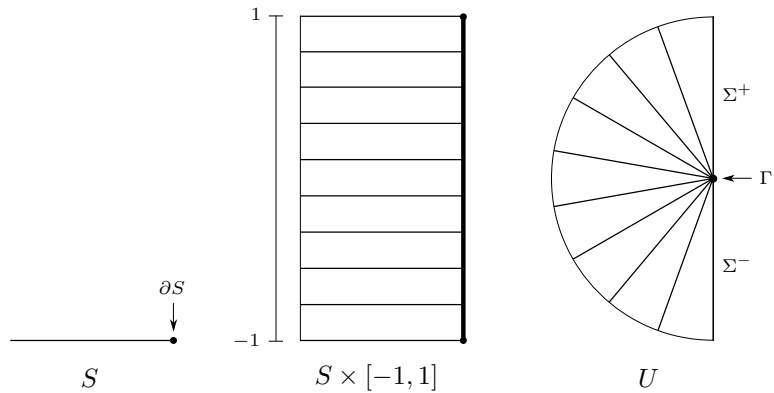}
\caption{\quad The diagram on the left represents the surface $S$. The diagram in the middle represents $S \times [-1,1]$; we have drawn some of the $S \times \{t\}$ fibers. The diagram on the right represents the handlebody $U$ obtained from $S \times [-1,1]$ by collapsing $\partial S \times [-1,1]$ to $\Gamma = \partial S \times \{0\}$.}
\label{fig:handle}
\end{center}
\end{figure}

It is not hard to prove that there exists a unique (up to isotopy) tight contact structure $\xi_0$ on $U$ for which $\Sigma$ is convex with dividing set $\Gamma$ (see \cite{et2}, for example). Let $(U_1,\xi_1)$ and $(U_2,\xi_2)$ be identical copies of $(U,\xi_0)$, with $\partial U_1 = \Sigma = \partial U_2$. 
According to Torisu \cite{tor}, $(M_{S,\phi},\xi_{S,\phi})$ is the contact 3-manifold obtained by gluing $(U_2,\xi_2)$ to $(U_1,\xi_1)$ via the orientation-reversing diffeomorphism $A_{\phi}:\partial U_2\rightarrow \partial U_1$ defined by 
$$A_{\phi}(x) = 
\left\{\begin{array}{lll} 
r(\phi(x)),& x \in \Sigma^+,\\ 
r(x),  & x \in \Sigma^-.
 \end{array} \right.$$
(The orientation on $M_{S,\phi}$ is specified by $M_{S,\phi} = U_1-U_2$.) The fact that $A_{\phi}$ sends $\Gamma \subset \partial U_1$ to $\Gamma\subset \partial U_2$ is what makes it possible to glue these two contact structures together, by Giroux's Flexibility Theorem \cite{giroux}.

Now suppose that $\phi$ is the composition $hg$. Let $I$ be the interval $[-\epsilon,\epsilon]$, and let $\xi_I$ be the $I$-invariant contact contact structure on $\Sigma \times I$ for which each $\Sigma \times \{t\}$ is convex with dividing set $\Gamma \times \{t\}$. Then $(M_{S,hg},\xi_{S,hg})$ may also be obtained by first gluing $(U_2,\xi_2)$ to $(\Sigma \times I, \xi_I)$ by the diffeomorphism from $\partial U_2$ to $\Sigma \times \{\epsilon\}$ which sends $x$ to $(A_g(x),\epsilon)$, and then gluing the resulting contact manifold to $(U_1,\xi_1)$ by the diffeomorphism from $\Sigma \times \{-\epsilon\}$ to $\partial U_1$ which sends $(x,-\epsilon)$ to $A_h( r(x))$. See Figure \ref{fig:handle2} for reference.

\begin{figure}[!htbp]
\labellist 
\hair 2pt 
\tiny
\pinlabel $\Sigma\times I$ at 472 105
\pinlabel $U_1$ at 360 136
\pinlabel $U_2$ at 585 136
\pinlabel $rg$ at 517 79
\pinlabel $r$ at 517 193
\pinlabel $rhr$ at 427 79
\pinlabel $id$ at 427 193
\pinlabel $\Sigma^+$ at 399 193
\pinlabel $\Sigma^+$ at 546 79
\pinlabel $\Sigma^-$ at 399 79
\pinlabel $\Sigma^-$ at 546 193
\pinlabel $U_1$ at 50 136
\pinlabel $U_2$ at 190 136
\pinlabel $r\phi$ at 122 79
\pinlabel $r$ at 122 193
\pinlabel $\Sigma^+$ at 93 193
\pinlabel $\Sigma^+$ at 153 79
\pinlabel $\Sigma^-$ at 93 79
\pinlabel $\Sigma^-$ at 153 193
\small
\pinlabel $-\epsilon$ at 435 258
\pinlabel $\epsilon$ at 502 258
\pinlabel $M_{S,\phi}$ at 122 3
\pinlabel $M_{S,hg}$ at 472 3
\endlabellist 
\begin{center}
\includegraphics[height=5.2cm]{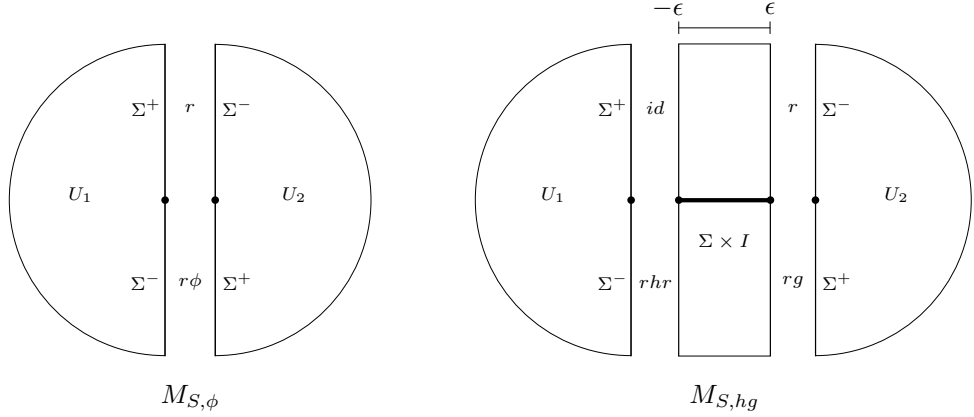}
\caption{\quad The diagram on the left illustrates the process of gluing $U_2$ to $U_1$ to form $M_{S,\phi}$. Alternatively, $M_{S,hg}$ can be formed by gluing $U_2$ to $\Sigma \times I$, and then gluing the result to $U_1$, as shown in the diagram on the right.}
\label{fig:handle2}
\end{center}
\end{figure}

If $S$ has genus $g$ and $r$ boundary components, then $\Sigma$ has genus $n=2g+r-1$. Let $b_1,\dots, b_n$ be disjoint, properly embedded arcs in $S$ for which $S\sminus\cup_i b_i$ is a disk. For $i=1,\dots,n$, we define the curve $\beta_i \subset \Sigma$ by $$\beta_i = b_i \times \{-1\} \cup b_i \times \{1\}.$$ (See Figure \ref{fig:arcs} for an example.) Note that $\beta_i$ bounds the attaching disk $b_i \times [-1,1]\subset U$. In particular, $U$ may be recovered from $\Sigma$ by thickening the surface, attaching 2-handles to one side along the curves $\beta_i$, and then gluing a 3-ball to the $S^2$ boundary component of the resulting manifold.

Let $\LLb$ be the link, contained in the $\Sigma \times I$ portion of $M_{S,hg}$, whose components are the curves $\beta_i \times \{0\} \subset \Sigma \times \{0\}$. The link $\LLb$ is \emph{nonisolating} in the convex surface $\Sigma \times \{0\}$; that is, $\LLb$ is transverse to $\Gamma \times \{0\}$, and the closure of every component of $\Sigma\times \{0\} \sminus (\Gamma \times \{0\} \cup \LLb)$ intersects $\Gamma \times \{0\}$. Therefore, by the Legendrian Realization Principle, we may assume that $\LLb$ is Legendrian \cite{honda2}. Moreover, each $\beta_i \times \{0\}$ intersects the dividing set $\Gamma \times \{0\}$ in exactly two places. It follows that $tw(\beta_i\times\{0\}, \Sigma \times \{0\})$, which measures the contact framing of $\beta_i\times\{0\}$ relative to the framing induced by the surface $\Sigma\times \{0\}$, is $$-\frac{1}{2}\,\#(\beta_i\times\{0\}\, \cap \,\Sigma\times\{0\}) = -1.$$ Therefore, contact $(+1)$-surgery on $\LLb$ is the same as $0$-surgery on $\LLb$ with respect to the framing induced by $\Sigma\times \{0\}$. 

\begin{figure}[!htbp]
\labellist 
\hair 2pt 
\tiny
\pinlabel $b_1$ at 173 230
\pinlabel $b_2$ at 226 230
\pinlabel $\beta_1$ at 536 230
\pinlabel $\beta_2$ at 589 230
\pinlabel $\Gamma$ at 632 287

\small
\pinlabel $\Sigma$ at 500 20
\pinlabel $S$ at 135 20
\endlabellist 
\begin{center}
\includegraphics[height = 5.6cm]{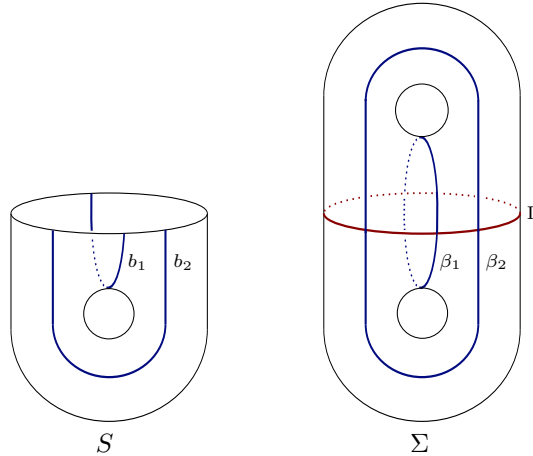}

\caption{\quad In this example, $S$ is a genus one surface with one boundary component. The diagram on the right shows the curves $\beta_1$ and $\beta_2$ in blue, and the dividing set $\Gamma$ in red.}
\label{fig:arcs}
\end{center}
\end{figure}

For any contact manifold $(M,\xi)$, and any Legendrian link $\LL \subset M$, let us denote by $(M_{\LL},\xi_{\LL})$ the contact manifold obtained from $M$ via contact $(+1)$-surgery on $\LL$.

\begin{proposition}
\label{prop:tight}
The contact manifold $((\Sigma\times I)_{\LLb}, (\xi_I)_{\LLb})$ is tight.
\end{proposition}

\begin{proof}
By construction, $((\Sigma \times I)_{\LLb}, (\xi_I)_{\LLb})$ embeds into $((M_{S,hg})_{\LLb},(\xi_{S,hg})_{\LLb})$ for any $h$ and $g$. So, it is enough to find an $h$ and $g$ for which the latter is tight. Let $h$ and $g$ each be the identity. In this case, $M_{S,hg}  = M_{S,id} \cong \#^{n}(S^1 \times S^2)$, and each component $\beta_i \times \{0\}$ of $\LLb$ bounds the disk $$b_i \times [-1,1]\cup \beta_i \times [-\epsilon,0] \,\subset\, U_1 \cup \Sigma \times [-\epsilon,0].$$ Moreover, the framings induced by these disks agree with the framings induced by $\Sigma \times \{0\}$. So, topologically, $(M_{S,id})_{\LLb}$ is the result of $0$-surgery on an $n$ component unlink in $\#^{n}(S^1 \times S^2)$; that is, $(M_{S,id})_{\LLb} \cong \#^{2n}(S^1 \times S^2)$. If $W$ is the cobordism from $M_{S,id}$ to $(M_{S,id})_{\LLb}$ obtained by attaching $0$-framed 2-handles to the unknots $\beta_i \times \{0\}$, then it follows that the map $$F_{-W}: \hf(-M_{S,id}) \rightarrow \hf(-(M_{S,id})_{\LLb})$$ is injective \cite[Proposition 6.1]{osz12}. By \cite{osz1}, this map sends $c(\xi_{S,id})$ to $c((\xi_{S,id})_{\LLb})$. The contact invariant $c(\xi_{S,id})$ is non-zero since $\xi_{S,id}$ is Stein fillable; hence, $c((\xi_{S,id})_{\LLb})$ is non-zero as well, by the injectivity of $F_{-W}$. Thus, $(\xi_{S,id})_{\LLb}$ is tight, and so is $(\xi_I)_{\LLb}$.
\end{proof}

\begin{proposition}
\label{prop:connectsum}
The contact manifold $((M_{S,hg})_{\LLb},(\xi_{S,hg})_{\LLb})$ is the contact connected sum $(M_{S,h},\xi_{S,h})\,\#\,(M_{S,g},\xi_{S,g})$.
\end{proposition}

\begin{proof}[Proof of Proposition \ref{prop:connectsum}] Let $N_i \subset \Sigma \times I$ be a tubular neighborhood of $\beta_i\times\{0\}$ such that $\partial N_i$ is the union of two annuli, $A^1_i \subset \Sigma \times [-\epsilon, 0]$ and $A^2_i \subset \Sigma \times [0, \epsilon]$. And let us think of $S^1$ as the union of two intervals, $S^1=I_1\cup I_2$. Topologically, $(\Sigma \times I)_{\LLb}$ is obtained from $\Sigma \times I$ by performing 0-surgery on $\LLb$ with respect to the framing induced by $\Sigma \times\{0\}$, as discussed above. $(\Sigma \times I)_{\LLb}$ is therefore the result of gluing solid tori $D^2_i\times S^1$ to $\Sigma \times I \sminus \cup_{i}\, \text{int}\,N_i$ so that $\partial D^2_i \times I_1$ is glued to $A^1_i$, and $\partial D^2_i \times I_2$ is glued to $A^2_i$ (see Figure \ref{fig:handle3}). So, $(\Sigma \times I)_{\LLb}$ is the union \begin{equation}\label{eqn:connect} (\Sigma \times [-\epsilon,0] \sminus \cup_{i}\, \text{int}\,N_i) \cup_i (D^2_i \times I_1)\,\, \bigcup \,\,(\Sigma \times [0,\epsilon] \sminus \cup_{i}\, \text{int}\,N_i) \cup_i (D^2_i \times I_2).\end{equation} Each of these two pieces is homeomorphic to the manifold obtained by thickening $\Sigma$ and attaching 2-handles to one side of this thickened surface along the curves $\beta_i$; in other words, each piece is the complement of a 3-ball in a genus $n$ handlebody, and these pieces are attached along their common $S^2$ boundary component. 

\begin{figure}[!htbp]
\labellist 
\hair 2pt 
\tiny
\pinlabel $\beta_i\times \{0\}$ at -15 134
\pinlabel $0$ at 60 259
\pinlabel $\Gamma\times \{0\}$ at 99 117
\pinlabel $A^1_i$ at 285 134
\pinlabel $D^2_i\times I_1$ at 381 46
\pinlabel $A^2_i$ at 567 134
\pinlabel $D^2_i\times I_2$ at 471 46
\small
\pinlabel $\Sigma\times I$ at 60 4
\pinlabel $-\epsilon$ at 16 260

\pinlabel $\epsilon$ at 98 260
\endlabellist 
\begin{center}
\includegraphics[height = 5.7cm]{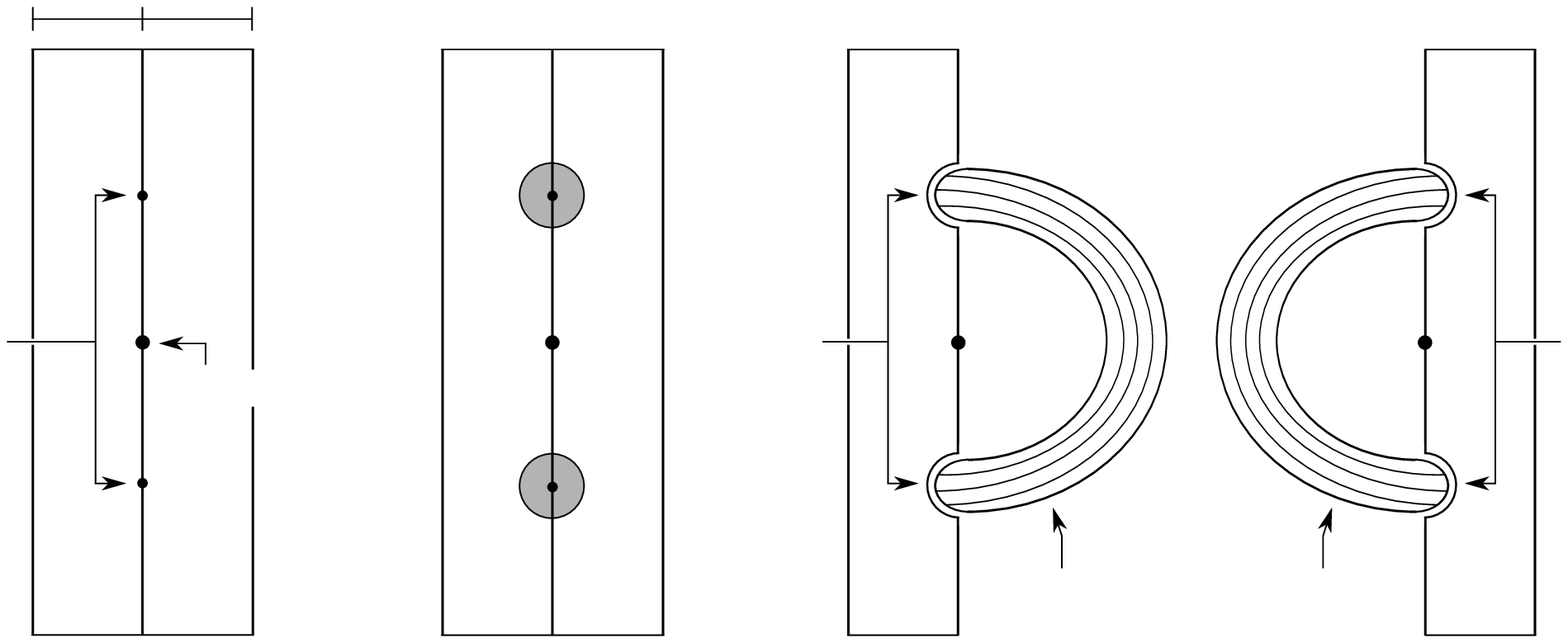}
\caption{\quad The diagram on the left shows the knot $\beta_i \times \{0\}\subset \Sigma \times \{0\}$. The shaded disks in the middle diagram represent the tubular neighborhood $N_i$. The diagram on the right illustrates the process of performing $0$-surgery on $\beta_i \times \{0\}$ by removing $N_i$ and gluing 2-handles along the annuli $A^1_i$ and $A^2_i$. We have drawn some of the $D^2 \times \{t\}$ fibers in these 2-handles.}
\label{fig:handle3}
\end{center}
\end{figure}

Let us denote the left and right pieces in (\ref{eqn:connect}) by $U_3 - B^3$ and $U_4 - B^3$, respectively, where $U_3$ and $U_4$ are genus $n$ handlebodies with $\partial U_3 = -\Sigma \times \{-\epsilon\}$ and $\partial U_4 = \Sigma \times \{\epsilon\}$. According to \cite{giroux}, their common $S^2$ boundary component can be made convex in $((\Sigma\times I)_{\LLb}, (\xi_I)_{\LLb})$ after a small isotopy. By Proposition \ref{prop:tight}, the restriction of $(\xi_I)_{\LLb}$ to $U_i-B^3$ is tight, for $i=3,4$. Therefore, by Honda's Gluing Theorem \cite[Theorem 2.5]{honda3}, the restriction $(\xi_I)_{\LLb}|_{U_i-B^3}$ is isotopic to the contact structure on the complement of a Darboux ball in $(U_i,\xi_i),$ where $\xi_i$ is the unique tight contact structure on $U_i$ for which $\partial U_i$ is convex with dividing set $\Gamma \times \{-\epsilon\}$ when $i=3$, and $\Gamma \times \{\epsilon\}$ when $i=4$. Said differently, $((\Sigma\times I)_{\LLb}, (\xi_I)_{\LLb})$ is the contact connected sum of identical copies, $(U_3,\xi_3)$ and $(U_4,\xi_4)$, of the contact handlebody $(U,\xi_0)$.

As a result, $((M_{S,hg})_{\LLb},(\xi_{S,hg})_{\LLb})$ may be pieced together as follows. First, glue $(U_2,\xi_2)$ to $(U_4,\xi_4)$ by the diffeomorphism from $\partial U_2$ to $\partial U_4 = \Sigma \times \{\epsilon\}$ which sends $x$ to $(A_g(x),\epsilon)$; this forms $(M_{S,g},\xi_{S,g})$. Next, glue $(U_3,\xi_3)$ to $(U_1,\xi_1)$ by the diffeomorphism from $-\partial U_3 = \Sigma \times \{-\epsilon\}$ to $\partial U_1$ which sends $(x,-\epsilon)$ to $A_h(r(x))$; this forms $(M_{S,h},\xi_{S,h})$. Finally, remove Darboux balls from the $U_3$ and $U_4$ portions of $M_{S,h}$ and $M_{S,g}$, and glue the resulting contact manifolds together by a diffeomorphism which identifies the dividing curves on their $S^2$ boundary components. This process realizes $((M_{S,hg})_{\LLb},(\xi_{S,hg})_{\LLb})$ as the contact connected sum of $(M_{S,h},\xi_{S,h})$ with $(M_{S,g},\xi_{S,g})$.
\end{proof}

\begin{proof}[Proof of Theorem \ref{thm:Stein}]
According to Ding and Geiges \cite[Proposition 8]{DG}, Proposition \ref{prop:connectsum} implies that $(M_{S,hg},\xi_{S,hg})$ is the result of contact $(-1)$-surgery on a link in the contact connected sum $(M_{S,h},\xi_{S,h})\,\#\,(M_{S,g},\xi_{S,g})$.
\end{proof}

\section{Fillability of cyclic branched covers}

The essential idea in the proof of Theorem \ref{thm:Stein} is that we can find curves on the convex surface $\Sigma\times\{0\}\subset M_{S,hg}$ which each intersect the dividing set twice and which are attaching curves for the handlebody $S \times [-1,1]$. These conditions guarantee that contact $(+1)$-surgery on these curves is the same as $0$-surgery with respect to the framing induced by $\Sigma$, and, therefore, that such surgery results in the appropriate connected sum. This idea can be applied more generally to prove results like Theorem \ref{thm:bnc}, as below.

\begin{proof}[Proof of Theorem \ref{thm:bnc}]
Suppose that $K$ is a transverse knot in a tight contact manifold $(M,\xi)$. If $S$ is a Seifert surface for $K$ for which $sl(K) =- \chi(S)$, then $S$ may be perturbed to be convex with dividing set $\Gamma$ disjoint from $\partial S$ \cite{bev}. This convex $S$ has an $I$-invariant neighborhood $N=S\times [-1,1]$ whose convex boundary, after rounding corners, is $\Sigma = DS$, the double of $S$. The dividing set on $\Sigma$ is given by $$\mathbf{\Gamma}=\Gamma\,\cup\,\bar{\Gamma}\,\cup\,C,$$ where $\Gamma$ is the dividing set of $S$ as it sits on $S\times\{1\}$, $\bar{\Gamma}$ is the dividing set of $S$ as it sits on $S\times\{-1\}$, and $C=\partial S\times\{1/2\}$ is a curve isotopic to $K$ (see the proof of \cite[Lemma 2.1]{bev}).

As before, let $b_1,\dots, b_n$ be disjoint, properly embedded arcs in $S$ for which $S\sminus\cup_i b_i$ is a disk, and define curves $\beta_i = \partial (b_i \times [-1,1])$ on $\Sigma$. Now, it is not necessarily true that each $\beta_i$ intersects $\mathbf{\Gamma}$ twice. To remedy this, we let $p_1,\dots,p_k$ denote the points of intersection between the $b_i$ and $\Gamma$, and consider instead the complementary handlebody $$N'=N\sminus \cup_j\, (\nu(p_j)\times[-1,1]),$$ where each $\nu(p_j)\times [-1,1]$ is a standard neighborhood of the Legendrian arc $p_j\times[-1,1]$ in $N$. The boundary $\Sigma'=\partial N'$ is obtained from $\Sigma$ by attaching $k$ tubes from $S\times\{1\}$ to $S\times\{-1\}$ corresponding to the $p_j$. Removing the $\nu(p_j)$ from $S$ cuts each $b_i$ into properly embedded arcs $b_{i,1},\dots,b_{i,n_i}$ in $S\sminus \cup_j \, \nu(p_j)$. Then the $\beta_{i,l} = \partial b_{i,l}\times[-1,1]$ are attaching curves for the handlebody $N'$ and can each be made to intersect the new dividing curves $\mathbf{\Gamma'}\subset \Sigma'$ twice (see Figure \ref{fig:complement}).

\begin{figure}[!htbp]
\labellist 
\hair 2pt 
\tiny
\pinlabel $-1$ at 415 18
\pinlabel $1$ at 415 121
\pinlabel $S\times\{1\}$ at 33 195
\pinlabel $S\times\{-1\}$ at 33 33
\pinlabel $\mathbf{\Gamma'}$ at 190 183
\pinlabel $\beta_{i,l}$ at 190 155
\pinlabel $\beta_{i,l-1}$ at 94 155
\pinlabel $\beta_{i,l+1}$ at 284 155

\endlabellist 

\begin{center}
\includegraphics[height = 4.2cm]{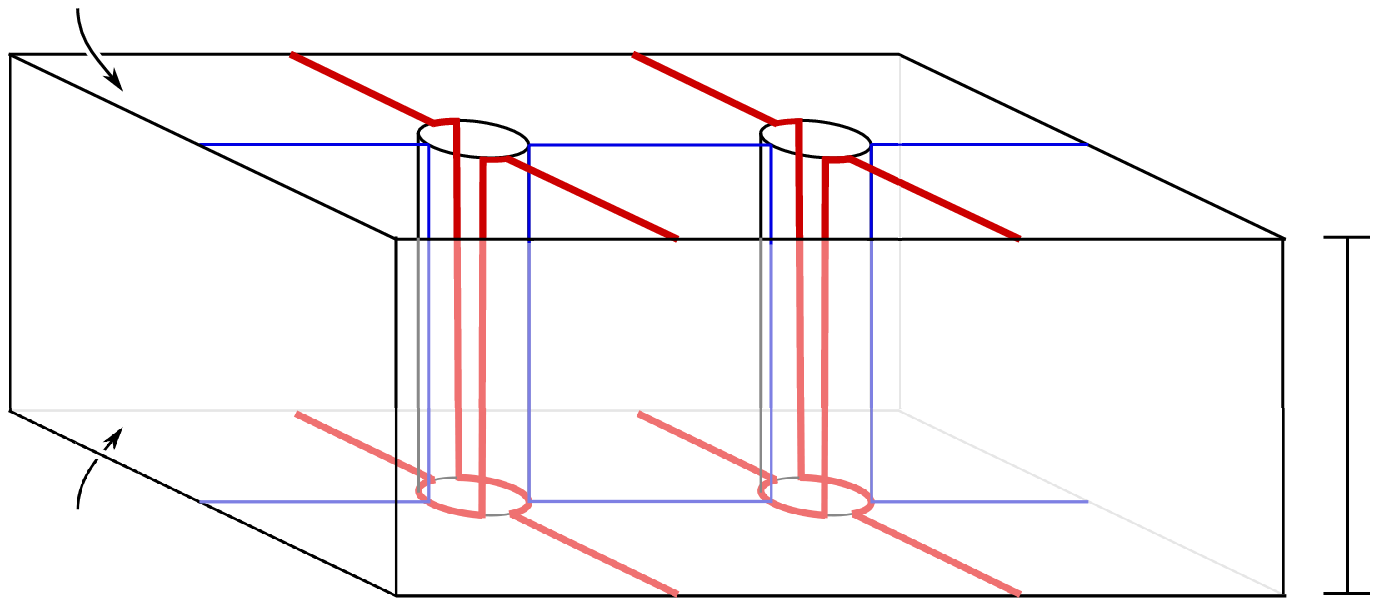}
\caption{\quad A portion of the handlebody $N'$, obtained by boring tunnels out of $N$. The dividing curve $\mathbf{\Gamma'}$ is in red. }
\label{fig:complement}
\end{center}
\end{figure}

Note that the contact manifold $(\Sigma^2(M,K),\xi^2(M,K))$ obtained by taking the double cover of $(M,\xi)$ branched along $K$ is formed by gluing together two copies of $M\setminus N$ along their boundaries (which are copies of $\Sigma$) so that the dividing curves match up. Therefore, by gluing together two copies of $M\setminus N'$ along their boundaries (which are copies of $\Sigma'$), one obtains a contact connected sum \begin{equation}\label{eqn:cs}(\Sigma^2(M,K),\xi^2(M,K))\,\#\,(\#^k(S^1\times S^2, \xi_0))\end{equation} for some contact structure $\xi_0$ on $\#^k(S^1\times S^2)$. Let $\beta_{i,l}'$ denote the copy of $\beta_{i,l}$ in this glued manifold, and let $$\mathbb{L} =\cup_{i=1}^n\cup_{l=1}^{n_i}\,\beta_{i,l}'.$$ 

Then contact $(+1)$-surgery on $\mathbb{L}$ produces the contact connected sum $(M,\xi)\,\#\,(M,\xi)$, as before. In other words, the manifold in (\ref{eqn:cs}) is obtained from $(M,\xi)\,\#\,(M,\xi)$ via $l$ Stein 2-handle additions. So, if $(M,\xi)$ is at least weakly fillable, then the same is true of the manifold in (\ref{eqn:cs}); in this case, $\xi_0$ must be the unique Stein fillable contact structure on $\#^k(S^1\times S^2)$.
Moreover, $(\Sigma^2(M,K),\xi^2(M,K))$ is obtained from the manifold in (\ref{eqn:cs}) by $k$ Stein 2-handle additions. So, in the end, $(\Sigma^2(M,K),\xi^2(M,K))$ is obtained from $(M,\xi)\,\#\,(M,\xi)$ via $(k+l)$ Stein 2-handle additions. As a result, we find that as long as $(M,\xi)$ is Stein (resp. strongly/weakly symplectically) fillable, then so is $(\Sigma^2(M,K),\xi^2(M,K))$. We can apply a similar construction for the lift of $K$ in $\Sigma^2(M,K)$ to conclude the analogous result for $(\Sigma^3(M,K),\xi^3(M,K))$, the 3-fold cyclic cover of $(M,\xi)$ branched along $K$, and so on.
\end{proof}

Theorem \ref{thm:bnc} ultimately rests on the fact that $K$ may be ``protected" from the dividing curves on $S$ whenever $sl(K) = -\chi(S)$  \cite{bev2}. That is, $S$ is isotopic to a convex surface with for which there is a component $C$ of the dividing set such that $C$ and $K$ cobound an annulus with characteristic foliation consisting of arcs from $C$ to $K$. This is what allows us to conclude that the dividing set on $\Sigma = DS$ is of the form $\Gamma\,\cup\,\bar{\Gamma}\,\cup\,C.$

It is an interesting problem to find a more general criterion which ensures that $K$ is the protected boundary of some Seifert surface. Van Horn-Morris and I hope to return to this problem in a future paper. For now, note that the proof of Theorem \ref{thm:bnc} provides an obstruction, as illustrated below.

\begin{proposition}
\label{prop:protect}
Let $B$ be the transverse 3-braid in $(S^3,\xi_{std})$ with braid word given by $(\sigma_1\sigma_2)^3\sigma_1\sigma_2^{-a_1}\dots\sigma_1\sigma_2^{-a_m},$ where the $a_i\geq 0$ and some $a_j\neq 0$. Then $B$ is not the protected boundary of any Seifert surface when $4+m-\sum a_i <0$. 
\end{proposition}

\begin{proof}[Proof of Proposition \ref{prop:protect}]
From the proof of Theorem \ref{thm:bnc}, it is enough to observe that the branched double cover $(\Sigma^2(S^3,B), \xi^2(S^3,B))$ is not Stein fillable when $4+m-\sum a_i <0$. This fact appears in \cite{bald3}.
\end{proof}

\begin{remark} One should contrast Proposition \ref{prop:protect} with the fact that $(\Sigma^n(S^3,B), \xi^n(S^3,B))$ is tight for all $n$ \cite{bald1}. 
\end{remark}

\bibliographystyle{hplain.bst}
\bibliography{References}

\end{document}